\documentclass[12pt,
reqno %comment this line out with a ``%'' if it bothers you
]{amsart}
\usepackage{amssymb,latexsym, amsmath, amscd, array, graphicx}

\swapnumbers
\numberwithin{equation}{section}

\theoremstyle{plain}

\newtheorem{thm}{Theorem}[section]
\newtheorem{theorem}[thm]{Theorem}

\newtheorem{lemma}[thm]{Lemma}

\newtheorem{prop}[thm]{Proposition}
\newtheorem{cor}[thm]{Corollary}

\newcommand\theoref{Theorem~\ref}
\newcommand\lemref{Lemma~\ref}
\newcommand\propref{Proposition~\ref}
\newcommand\corref{Corollary~\ref}

\def\secref{Section~\ref}

\theoremstyle{definition}

\newtheorem{definition}[thm]{Definition}

\newtheorem{rem}[thm]{Remark}
\newtheorem{remark}[thm]{Remark}

\newtheorem{question}[thm]{Question}

\DeclareMathOperator{\cat}{{\mbox{\rm cat$_{\rm LS}$}}}
\DeclareMathOperator{\syscat}{{\rm cat_{sys}}}

\DeclareMathOperator{\cd}{{\rm cd}}

\DeclareMathOperator{\stsys}{{\rm stsys}}
\DeclareMathOperator{\sys}{{\rm sys}}
\DeclareMathOperator{\sysh}{{\rm sysh}}
\DeclareMathOperator{\pisys}{{\rm sys}\pi}

\DeclareMathOperator{\vol}{{\rm vol}}

\def\ga{\alpha}
\def\gb{\beta}

\def\C{{\mathbb C}}
\def\Z{{\mathbb Z}}

\def\R{{\mathbb R}}
\def\N{{\mathbb N}}

\def\ds{\displaystyle}

\def\1{\hbox{\rm\rlap {1}\hskip.03in{\rom I}}}
\def\Bbbone{{\rm1\mathchoice{\kern-0.25em}{\kern-0.25em}
{\kern-0.2em}{\kern-0.2em}I}}

\def\ds{\displaystyle}

\def\ov{\overline}

\def\la{\langle}
\def\ra{\rangle}
\def\cupr{\smallsmile}

\long\def\forget#1\forgotten{} %

\newcommand{\gmetric}{{\mathcal G}}
\newcommand\gm{\gmetric}

\newcommand{\regular}{{B}}

\newcommand\ver[1]{\marginpar{\tiny Changed in Ver \VER}}

\newcommand\T {{\mathbb T}}
\newcommand\AJ {{\mathcal A}}
\DeclareMathOperator{\fmanifold}{{[\overline{\it F}_{\it \! M}]}}
\newcommand\manbar {{\overline{{M}}}}
\newcommand{\surface}{W}

\date{\today}
\begin{document}

\title[Cohomological dimension, self-linking, and systoles]
{Cohomological dimension, self-linking, and systolic geometry}

\author[A.~Dranishnikov]{Alexander N. Dranishnikov$^{1}$} %
\thanks{$^{1}$Supported by the NSF, grant DMS-0604494}

\author[M.~Katz]{Mikhail G. Katz$^{2}$} %
\thanks{$^{2}$Supported by the Israel Science Foundation (grants 84/03
and 1294/06) and the BSF (grant 2006393)}

\author[Yu.~Rudyak]{Yuli B. Rudyak$^{3}$}%
\thanks{$^{3}$Supported by the NSF, grant 0406311}

\address{Alexander N. Dranishnikov, Department of Mathematics, University
of Florida, 358 Little Hall, Gainesville, FL 32611-8105, USA}
\email{dranish@math.ufl.edu}

\address{Mikhail G. Katz, Department of Mathematics, Bar Ilan
University, Ramat Gan 52900 Israel}

\email{katzmik ``at'' math.biu.ac.il}

\address{Yuli B. Rudyak, Department of Mathematics, University
of Florida, 358 Little Hall, Gainesville, FL 32611-8105, USA}
\email{rudyak@math.ufl.edu}



\begin{abstract}
Given a closed manifold~$M$, we prove the upper bound of
\begin{equation*}
\tfrac{1}{2} (\dim M+ \cd(\pi_1 M))
\end{equation*}
for the number of systolic factors in a curvature-free lower bound for
the total volume of~$M$, in the spirit of M.~Gromov's systolic
inequalities.  Here ``$\cd$'' is the cohomological dimension.  We
apply this upper bound to show that, in the case of a~$4$-manifold,
the Lusternik--Schnirelmann category is an upper bound for the
systolic category.  Furthermore we prove a systolic inequality on a
manifold~$M$ with~$b_1(M)=2$ in the presence of a nontrivial
self-linking class of a typical fiber of its Abel--Jacobi map to
the~$2$-torus.
\end{abstract}

\maketitle

\tableofcontents

\section{Systolic inequalities and LS category}

A quarter century ago, M.~Gromov \cite{Gr1} initiated the modern
period in systolic geometry by proving a curvature-free~$1$-systolic
lower bound for the total volume of an essential Riemannian
manifold~$M$ of dimension~$d$, i.e.~a~$d$-fold product of the systole
is such a lower bound.

Here the term ``curvature-free'' is used in the literature to refer to
a bound independent of curvature invariants, with a constant depending
on the dimension (and possibly on the topology), but not on the
geometry (i.e.~the Riemannian metric).  Note that such bounds cannot
be called ``metric-independent'' as the systolic invariants themselves
do depend on the metric.

Recently, M.~Brunnbauer~\cite{Bru3} proved that a~$(k-1)$-connected
manifold of dimension~$n=kd$ satisfies a curvature-free
stable~$k$-systolic inequality~$(\stsys_k)^d \leq C \vol_n$ if and
only if a purely homotopic condition on the image of the fundamental
class~$[M]$ in a suitable classifying space is satisfied.  Thus the
total volume admits a lower bound in terms of a systolic product
with~$d$ factors.  J.~Strom~\cite{S2} and the first named
author~\cite{D} were motivated by investigations of the
Lusternik-Schnirelmann (LS) category \cite{DKR} in its relation to
systolic geometry~\cite{KR1}.  A.~Costa and M.~Farber \cite{CF} have
pursued the direction initiated in \cite{D}, as applied to motion
planning--related complexity.

Our first result links systolic inequalities to the cohomological
dimension (see \cite{Bro}) of the fundamental group, see
\theoref{t:cd} for a more precise statement.

\begin{theorem}
\label{aaa}
Let~$M$ be a closed~$n$-manifold, with~$\cd(\pi_1 M)=d \leq n$.  Then
there can be no more than
\[
\frac{n+d}{2}
\]
systolic factors in a product which provides a curvature-free lower
bound for the total volume of~$M$.
\end{theorem}

It was shown in~\cite{KR1, KR2} that the maximal number of factors in
such a product coincides with the LS category~$\cat$ in a number of
cases, including all manifolds of dimension~$\leq 3$.  We apply
Theorem~\ref{aaa} to show that, in dimension~$4$, the number of
factors is bounded by the LS category, see \corref{c:dim=4} for a more
precise statement.

\begin{theorem}
\label{bbb}
For every closed orientable~$4$-manifold, the maximal number of
factors in a product of systoles which provides a curvature-free lower
bound for the total volume, is bounded above by~$\cat M$.
\end{theorem}

Combining Theorem~\ref{aaa} with a volume lower bound resulting from
an inequality of Gromov's (see Sections~\ref{s:def} and
\ref{s:abjac}), we obtain the following result (typical examples are
$\T^2 \times S^3$ as well as the non-orientable~$S^3$-bundle over
$\T^2$).

\begin{theorem}
\label{ccc}
Let~$M$ be a closed~$5$-dimensional manifold.  Assume that~$b_1( M)=
\cd(\pi_1 M)=2$ and furthermore that the typical fiber of the
Abel--Jacobi map to~$\T^2$ represents a nontrivial homology class.
Then the maximal possible number of factors in a systolic lower bound
for the total volume is~$3$.  Note also that~$3\le \cat M \le 4$.
\end{theorem}

The above result motivates the following question concerning upper
bounds for the Lusternik-Schnirelmann category, cf.~\cite{D}.

\begin{question}
Under the hypotheses of Theorem~\ref{ccc}, is the
Lusternik-Schnirelmann category of $M$ necessarily equal to $3$?
\end{question}

It will be convenient to formulate all of the above results in terms
of the systolic category.  The idea of systolic category is to codify
lower bounds for the total volume, in terms of lower-dimensional
systolic invariants.  We think of it as an elegant way of expressing
systolic statements.  Here we wish to incorporate all possible
curvature-free systolic inequalities, stable or unstable.  More
specifically, we proceed as follows.

\begin{definition}
\label{21b}
Given~$k\in \N, k>1$ we set
\begin{equation*}
\ \sys_{k}(M, \gmetric)= \inf \left\{\sysh_k^{\phantom{I}}(M',
\gmetric;A), \stsys_k(M, \gmetric) \right\},
\end{equation*}
where $\sysh$ is the homology systole, $\stsys$ is the stable homology
systole, and the infimum is over all regular covering spaces~$M'$
of~$M$, and over all choices
\begin{equation}
\label{41}
A\in \{ \Z, \Z_2 \} .
\end{equation}
Furthermore, we define
\begin{equation*}
\sys_{1}(M, \gmetric)=\min \{\pisys_1(M, \gmetric),\stsys_1(M,
\gmetric)\}.
\end{equation*}
\end{definition}

Note that the systolic invariants thus defined are nonzero~\cite{KR2}.

\begin{definition}
%\label{d:syscat}
Let~$M$ be a closed~$n$-dimensional manifold.  Let~$d\geq 1$ be an
integer.  Consider a partition
\begin{equation}
\label{eq:partition}
n= k_1 + \ldots + k_d,\quad k_1\le k_2\le \cdots \le k_d
\end{equation}
where~$k_i\geq 1$ for all~$i=1,\ldots, d$.  We say that the partition
(or the~$d$-tuple~$(k_1, \ldots, k_d)$) is {\em categorical} for~$M$
if the inequality
\begin{equation}
\label{eq:main}
\sys_{k_1}(\gmetric) \sys_{k_2}(\gmetric) \ldots \sys_{k_d}(\gmetric)
\leq C(M) \vol_n(\gmetric)
\end{equation}
is satisfied by all metrics~$\gmetric$ on~$M$, where the
constant~$C(M)$ is expected to depend only on the topological type
of~$M$, but not on the metric~$\gmetric$.
\end{definition}

The {\em size\/} of a partition is defined to be the integer~$d$.

\begin{definition}
The {\em systolic category} of~$M$, denoted~$\syscat(M)$, is the
largest size of a categorical partition for~$M$.
\end{definition}

In particular, we have~$\syscat M \le \dim M$.

We know of no example of manifold whose systolic category exceeds its
Lusternik-Schnirelmann category.  The lower bound of~$b_1(M)+1$ for the
systolic category of a manifold~$M$ with non-vanishing fiber class in
the free abelian cover of~$M$, discussed in Section~\ref{s:abjac},
therefore inspires the following question.

\begin{question}
Is the non-vanishing of the fiber class in the free abelian cover
of~$M$, a sufficient condition to guarantee a lower bound of
$b_1(M)+1$ for the Lusternik-Schnirelmann category of~$M$?
\end{question}

The answer is affirmative if the fiber class can be represented as a
Massey product, see Section~\ref{s:abjac}.

The paper is organized as follows.  In Sections \ref{s:prel} and
\ref{four}, we review the notion of systolic category.  In
\secref{s:group}, we obtain an upper bound for the systolic category
of a closed~$n$-manifold in terms of its fundamental group and, in
particular, prove that the systolic category of a 4-manifold with free
fundamental group does not exceed~$2$ (\corref{c:free}).

In \secref{s:abjac} we investigate the next possible value of the
categories, namely~$3$.  We recall a 1983 result of M.~Gromov's
related to Abel--Jacobi maps, and apply it to obtain a lower bound
of~$3$ for the systolic category for a class of manifolds defined by a
condition of non-trivial self-linking of a typical fiber of the
Abel--Jacobi map. In fact, non-triviality of the self-linking class
guarantees the homological non-triviality of the typical fiber lifted
to the free abelian cover.
%
% , and thus Theorem~\ref{ccc} is applicable. NO IT IS NOT.
%

Marcel Berger's monograph \cite[ pp.~325-353]{Be6} contains a detailed
exposition of the state of systolic affairs up to '03.  More recent
developments are covered in \cite{SGT}.

Recent publications in systolic geometry include the articles
\cite{AK, Be08, Bru, Bru2, Bru3, BKSW, DKR, DR09, EL, HKU, KK, KK2,
KW, Ka4, KSh, RS, Sa08}.

\section{A systolic introduction}
\label{s:prel}

Let~$\T^2$ be a 2-dimensional torus equipped with a Riemannian
metric~$\gm$.  Let~$A$ be the area of~$(\T^2,\gm)$.  Let~$\ell$ be the
least length of a noncontractible loop in~$(\T^2,\gm)$.  What can one
say about the scale-invariant ratio~$\ell^2/A$?  It is easy to see
that the ratio can be made arbitrarily small for a suitable choice of
metric~$\gm$.  On the other hand, it turns out that the ratio is
bounded from above.  Indeed, C.~Loewner proved that~$\ell^2/A\le
2/\sqrt 3$, for all metrics~$\gm$ on~$\T^2$, see~\cite{Pu}.

More generally, given a closed~$n$-dimensional smooth manifold~$M$
with non-trivial fundamental group, M. Berger and M. Gromov asked
whether there exists a constant~$C>0$ such that the inequality
\begin{equation}\label{e:gromov}
\ell^n\le C\vol(M)=C\vol_n(M,\gm)
\end{equation}
holds for all Riemannian metrics~$\gm$ on~$M$; here~$\ell$ is the
least length of a noncontractible loop in~$M$, while~$C$ is required
to be independent of the metric~$\gm$.  Indeed, Gromov~\cite{Gr1}
proved that such a~$C=C_n$ exists if~$M$ is an essential manifold,
meaning that~$M$ represents a nonzero homology class
in~$H_n(\pi_1(M))$.  I.~Babenko~\cite{Bab1} proved a converse.

We generalize these invariants as follows.  Let~$h_k$ denote the
minimum of~$k$-volumes of homologically nontrivial~$k$-dimensional
cycles (see Section~\ref{s:prel} for details).  Do there exist a
partition~$k_1+ \cdots +k_d=n$ of~$n$, and a constant~$C$ such that
the inequality
\begin{equation}\label{e:syscat}
\prod_{i=1}^d h_{k_i}\le C\vol(M, \gm)
\end{equation}
holds for all Riemannian metrics~$\gm$ on~$M$?  The invariants~$h_k$
are the~$k$-systoles, and the maximum value of~$d$ as in
\eqref{e:syscat} is called the {\em systolic category\/}~$\syscat$
of~$M$ \cite{KR1,SGT}.  The goal of this work is to continue the
investigation of the invariant~$\syscat$ started in \cite{KR1, KR2}.

%
%technical material moved to appendix
%

It was originally pointed out by Gromov (see \cite{Be5}) that this
definition has a certain shortcoming.  Namely, for~$S^1\times S^3$ one
observes a ``systolic freedom'' phenomenon, in that the inequality
\begin{equation}\label{e:freedom}
\sysh_1\sysh_3\le C\vol(S^1\times S^3)
\end{equation}
is violated for any~$C\in \R$, by a suitable metric~$\gm$
on~$S^1\times S^3$ \cite{Gr2,Gr3}.  This phenomenon can be overcome by
a process of stabilisation.

It turns out that the difference between the stable
systoles,~$\stsys_k$, and the ordinary ones,~$\sysh_k$ has significant
ramifications in terms of the existence of geometric inequalities.
Namely, in contrast with the violation of \eqref{e:freedom}, the
inequality
\begin{equation}
\label{27}
\stsys_1\stsys_3\le \vol_4(S^1\times S^3)
\end{equation}
holds for all metrics on~$S^1\times S^3$ \cite{Gr1} (see
Theorem~\ref{t:cuplength} for a generalisation).

\section{Gromov's inequalities}
\label{s:def}
\label{four}

Systolic category can be thought of as a way of codifying three
distinct types of systolic inequalities, all due to Gromov.  There are
three main sources of systolic lower bounds for the total volume of a
closed manifold~$M$.  All three originate in Gromov's 1983 Filling
paper~\cite{Gr1}, and can be summarized as follows.

\begin{enumerate}
\item
Gromov's inequality for the homotopy~$1$-systole of an essential
manifold~$M$, see \cite{We, Gu09} and \cite[p.~97]{SGT}.
\item
Gromov's stable systolic inequality (treated in more detail in
\cite{BK1, BK2}) corresponding to a cup product decomposition of the
rational fundamental cohomology class of~$M$, see
\theoref{t:cuplength}.
\item
A construction using the Abel--Jacobi map to the Jacobi torus of~$M$
(sometimes called the dual torus), also based on a theorem of Gromov
(elaborated in \cite{IK, BCIK2}).
\end{enumerate}

Let us describe the last construction in more detail.  Let~$M$ be a
connected~$n$-manifold.  Let~$b=b_1(M)$.  Let
\begin{equation*}
\T^b := H_1(M;\R)/H_1(M;\Z)_\R
\end{equation*}
be its Jacobi torus.  A natural metric on the Jacobi torus of a
Riemannian manifold is defined by the stable norm, see
\cite[p.~94]{SGT}.

The Abel--Jacobi map~$\AJ_M: M\to \T^b$ is discussed in \cite{Li,
BK2}, cf.~\cite[p.~139]{SGT}.  A typical fiber~$F_M\subset M$
(i.e.~inverse image of a generic point) of~$\AJ_M$ is a smooth
imbedded~$(n-b)$-submanifold (varying in type as a function of the
point of~$\T^b$).  Our starting point is the following observation of
Gromov's \cite[Theorem~7.5.B]{Gr1}, elaborated in \cite{IK}.

\begin{theorem}[M.~Gromov]
\label{t:abcover}
If the homology class~$\fmanifold\in H_{n-b}(\ov M)$ of the lift
of~$F_M$ to the maximal free abelian cover~$\ov M$ of~$M$ is nonzero,
then the total volume of~$M$ admits a lower bound in terms of the
product of the volume of the Jacobi torus and the infimum of areas of
cycles representing the class~$\fmanifold$.
\end{theorem}

%
%technical material moved to appendix
%

We also reproduce the following result, due to Gromov~\cite{Gr1}, see
also~\cite{BK1, SGT}, stated in terms of systolic category.

\begin{thm}[M.~Gromov]
\label{t:cuplength}
For a closed orientable manifold~$M$, the systolic category of~$M$ is
bounded from below by the rational cup-length of~$M$.
\end{thm}

Note that Theorem~\ref{t:cuplength} is not directly related to
Theorem~\ref{t:abcover}.  For instance, Theorem~\ref{t:cuplength}
implies that the systolic category of~$S^2\times S^2$, or~$\C {\mathbb
P} ^2$, equals~$2$, while Theorem~\ref{t:abcover} gives no information
about simply-connected manifolds.

\section{Fundamental group and systolic category}
\label{s:group}

Throughout this section we will assume that~$\pi$ is a finitely
presented group.  Denote by~$\cd(\pi)$ the cohomological dimension
of~$\pi$, see \cite{Bro}.

\begin{lemma}
\label{l:classif}
If~$\pi_1(M)$ is of finite cohomological dimension then~$M$ admits a
map~$f: M \to B$ to a simplicial complex~$B$ of dimension at most
$\cd(\pi_1 M)$, inducing an isomorphism of fundamental groups.
\end{lemma}

\begin{proof}
Let~$\pi=\pi_1(M)$.  When~$d=\cd(\pi)\not=2$, by the Eilenberg--Ganea
theorem~\cite{EG}, there exists a~$d$-dimensional model~$B\pi$ for the
classifying space of~$\pi$, which may be assumed to be a simplicial
complex.  This yields the desired map~$f: M \to B\pi$ by universality
of the classifying space.

In the case~$d=2$ it is still unknown whether one can assume that
$\dim B \pi=2$ for all~$\pi$.  That this is so is the content of the
Eilenberg--Ganea conjecture.  In this case we can assume only
that~$\dim B\pi\le 3$.

We claim that there is a map~$q:B\pi\to B\pi^{(2)}$ onto the
$2$-skeleton that induces an isomorphism of the fundamental
groups. Indeed, the obstruction to retracting~$B\pi$ onto~$B\pi^{(2)}$
is an element of the cohomology group
\[
H^3 \left(B\pi;\pi_2\left(B\pi^{(2)}\right)\right)
\]
with coefficients in the~$\pi$-module~$\pi_2(B\pi^{(2)})$. This group
is zero, since by hypothesis~$cd(\pi)=2$.  By the classical
obstruction theory the identity map of~$B\pi^{(2)}$ can be changed on
the 2-dimensional skeleton without changes on the 1-dimensional
skeleton in such a way that a new map has an extension~$q$
to~$B\pi$. Since~$q:B\pi\to B\pi^{(2)}$ is the identity on the
1-skeleton, it induces an isomorphism of the fundamental groups.

Now in the case~$d=2$ we use the complex~$B\pi^{(2)}$ instead of
$B\pi$ together with the map~$q\circ f:M\to B\pi^{(2)}$ instead of
$f$.
\end{proof}

\begin{thm}
\label{t:cd}
Let~$M$ be a closed~$n$-manifold, with~$\cd(\pi_1(M))=d \leq n$.  Then
the systolic category of~$M$ is at most~$(n+d)/2$.
\end{thm}

\begin{proof}
Let~$g_M$ be a fixed background metric on~$M$. Choose a map
\begin{equation*}
f: M \to B
\end{equation*}
as in \lemref{l:classif}, as well as a fixed PL metric~$g_B$ on~$B$.
Consider the metric~$f^*(g_B)$ on~$M$ pulled back by~$f$ (see
\cite{Bab1, SGT}).  This metric is defined by a quadratic form of rank
at most~$d$ at every point of~$M$.  Note that the quadratic form can
be thought of as the {\em square\/} of the length element.

Next, we scale the pull-back metric by a large real parameter~$t\in
\R$.  When the length of a vector is multiplied by~$t$, the convention
is to write the metric as~$t^2 f^* (g_B)$.  Since the rank of the
metric is at most~$d$, the volume of~$g_M + t^2 f^* (g_B)$ grows at
most as~$t^d$, where~$g_M$ is any fixed background metric on~$M$.  We
obtain a family of metrics
\begin{equation}
\label{51}
g_t= g_M + t^2 f^*(g_B)
\end{equation}
on~$M$ with volume growing at most as~$t^d$ where~$d=\dim B$, while
the~$1$-systole grows as~$t$.  Thus the ratio
\begin{equation*}
\frac{\sys_1^{d+1}(M,g_t)}{\vol(M,g_t)}
\end{equation*}
tends to infinity.  The addition of a fixed background metric on~$M$
in~\eqref{51} ensures a uniform lower bound for all its~$k$-systoles
for~$k\geq 2$.  It follows that a partition
\begin{equation*}
n= \underset{d+1}{\underbrace{1+\dots+1}} +k_{d+2}+\dots+k_r
\end{equation*}
cannot be categorical.
\end{proof}

\section{Systolic category, LS category, and~$\cd(\pi)$}

\begin{rem}
\label{sharp}
The upper bound of Theorem~\ref{t:cd} is sharp when~$\pi$ is a free
group ($d=1$) in the sense that~$\pi$ is the fundamental group of a
closed~$(2k+1)$-manifold
\begin{equation*}
M=\left( \prod_{i=1}^{k-1}S^2 \right) \times (\#(S^1\times S^2))
\end{equation*}
with~$\syscat M=k+1= (2k+2)/2$ in view of Theorem~\ref{t:cuplength}.
Similarly, for the~$(2k+2)$-dimensional manifold
\begin{equation*}
M=S^3\times\prod_{i=1}^{k-2}S^2\times (\#(S^1\times S^2)),
\end{equation*}
where~$k>1$, we have~$\syscat M=k+1$.  The case of a~$4$-dimensional
manifold follows from \corref{c:free}.
\end{rem}

%\begin{rem}
%\label{r:cd=n}
%For~$\cd \pi=n=\dim M$ we have~$\syscat M =\cat M=n$,~\cite{KR1, SGT}.
%\end{rem}

\begin{cor}
\label{c:free}
The systolic category of a closed orientable~$4$-manifold with free
fundamental group is at most~$2$, and it is exactly~$2$ if~$M$ is not
a homotopy sphere.
\end{cor}

\begin{proof}
The partitions~$4=1+1+2$ and~$4=1+1+1+1$ are ruled out by
Theorem~\ref{t:cd}.  The only remaining possibilities are the
partitions~$4=1+3$ (corresponding to category value~$2$) and~$4=4$
(that would correspond to category value~$1$).  If~$M$ is not
simply-connected, then by the hypothesis of the corollary,~$b_1(M)\geq
1$ and therefore~$M$ satisfies a systolic inequality of
type~\eqref{27} (see Theorem~\ref{t:cuplength}), proving that
$\syscat(M)=2$.  If~$M$ is simply-connected but not a homotopy sphere,
then~$H_2(M)=\pi_2(M)$ is free abelian and~$b_2(M)>0$, so that~$M$
satisfies the systolic inequality
\[
\stsys_2(M)^2 \leq b_2(M) \vol(M),
\]
see \cite{BK1} (a special case of Theorem~\ref{t:cuplength}),
proving~$\syscat(M)=2$.
\end{proof}

\begin{cor}
\label{c:dim=4}
For every closed orientable~$4$-manifold we have the
inequality~$\syscat M \le \cat M$, and the strict inequality is
possible in the case~$\syscat M=2<3=\cat M$ only.
\end{cor}

We do not know, however, if the case of the strict inequality can be
realized.

\begin{proof}
If the fundamental group of~$M$ is free then~$\cat M \le 2$~\cite{MK},
and in this case the result follows from \corref{c:free}. If~$\syscat
M=4$ then~$\cat M=4$,~\cite{KR1, SGT}, and hence~$\syscat M \le \cat
M$ if~$\cat M \ge 3$.  Finally, if the fundamental group of~$M$ is not
free then~$\cat M \ge 3$ \cite{DKR}.
\end{proof}

\begin{rem}
\label{r:low}
The equality~$\syscat M^n = \cat M^n$ for~$n\le 3$ was proved in
\cite{KR1,KR2}.
\end{rem}

\section{Self-linking of fibers and a lower bound for~$\syscat$}
\label{s:abjac}

In this section we will continue with the notation of
\theoref{t:abcover}.

\begin{prop}
\label{p:fm}
Let~$M$ be a closed connected manifold (orientable or non-orientable).
If a typical fiber of the Abel--Jacobi map represents a nontrivial
$(n-b)$-dimensional homology class in~$\ov M$, then systolic category
satisfies~$\syscat(M)\geq b+1$.
\end{prop}

\begin{proof}
If the fiber class is nonzero, then the Abel--Jacobi map is
necessarily surjective in the set-theoretic sense.  One then applies
the technique of Gromov's proof of Theorem~\ref{t:abcover},
cf.~\cite{IK}, combined with a lower bound for the volume of the
Jacobi torus in terms of the~$b$-th power of the stable~$1$-systole,
to obtain a systolic lower bound for the total volume corresponding to
the partition
\begin{equation*}
n=1+1+\cdots+1+(n-b),
\end{equation*}
where the summand ``$1$'' occurs~$b$ times.  Note that Poincar\'e
duality is not used in the proof.
\end{proof}

%Note that Theorem~\ref{ccc} follows by combining \propref{p:fm} with
%Theorem~\ref{aaa}.
%there is a problem in the nonorientable case

The goal of the remainder of this section is to describe a sufficient
condition for applying Gromov's theorem, so as to obtain such a lower
bound in the case when the fiber class in~$M$ vanishes.

From now on we assume that~$M$ is orientable, has dimension~$n$,
and~$b_1(M)=2$.  Let~$\{\alpha,\beta \} \subset H^1(M)$ be an integral
basis for~$H^1(M)$.  Let~$F_M$ be a typical fiber of the Abel--Jacobi
map.  It is easy to see that~$[F_M]$ is Poincar\'e dual to the cup
product~$\ga\cupr \gb$.  Thus, if~$\ga\cupr\gb\ne 0$ then~$\syscat M\ge
3$ by \propref{p:fm}.  If~$\ga\cupr \gb=0$ then the Massey
product~$\la\ga,\ga,\gb\ra$ is defined and has zero indeterminacy.

\begin{thm}\label{t:massey}
Let~$M$ be a closed connected orientable manifold of dimension~$n$
with~$b_1(M)=2$.  If~$\la\ga,\ga,\gb\ra\cupr\gb \ne 0$ then~$\syscat M
\ge 3$.
\end{thm}

Note that, on the Lusternik--Schnirelmann side, we similarly have a
lower bound~$\cat M \ge 3$ if~$\la\ga,\ga,\gb\ra\ne 0$, since the
element~$\la \ga, \ga, \gb \ra$ has category weight~$2$~\cite{R1,R2}.

To prove the theorem, we reformulate it in the dual homology language.

\begin{definition}
 Let~$F=F_M \subset M$ be an oriented typical fiber.  Assume~$[F]=0\in
H_{n-2}(M)$.  Choose an~$(n-1)$-chain~$X$ with~$\partial X = F$.
Consider another regular fiber~$F'\subset M$.  The oriented
intersection~$X\cap F'$ defines a class
\begin{equation*}
\ell_M(F_M, F_M) \in H_{n-3}(M),
\end{equation*}
which will be referred to as the {\em self-linking class\/} of a
typical fiber of~$\AJ_M$.
\end{definition}

The following lemma asserts, in particular, that the self-linking
class is well-defined, at least up to sign.

\begin{lemma}
The class~$\ell_M(F_M, F_M)$ is dual, up to sign, to the cohomology
class~$\la \alpha,\alpha,\beta \ra \cupr\beta \in H^3(M)$.
\end{lemma}

\begin{proof}
The classes~$\alpha, \beta$ are Poincar\'e dual to hypersurfaces~$A, B
\subset M$ obtained as the inverse images under~$\AJ_M$ of a
pair~$\{u,v\}$ of loops defining a generating set for~$H_1(\T^2)$.
The hypersurfaces can be constructed as inverse images of a regular
point under a projection~$M\to S^1$ to one of the summand circles.
Clearly, the intersection~$A\cap B \subset M$ is a typical fiber
\begin{equation*}
F_M=A\cap B
\end{equation*}
of the Abel--Jacobi map (namely, inverse image of the point~$u\cap
v\in \T^2$). Then another regular fiber~$F'$ can be represented
as~$A'\cap B'$ where, say, the set~$A'$ is the inverse image of a
loop~$u'$ ``parallel'' to~$u$. Then~$A'\cap X$ is a cycle, since
\begin{equation*}
\partial (A'\cap X)=A'\cap A \cap B=\emptyset.
\end{equation*}
Moreover, it is easy to see that the homology class~$[A'\cap X]$ is
dual to the Massey product~$\la \ga,\ga, \gb\ra$, by taking a
representative~$a$ of~$\ga$ such that~$a\cupr a=0$.  Now,
since~$F'=A'\cap B'$, we conclude that~$[F'\cap X]$ is dual, up to
sign, to~$\la \ga,\ga,\gb\ra\cupr \gb$.
\end{proof}

\begin{remark}
In the case of~$3$-manifolds with first Betti number~$2$, the
non-vanishing of the self-linking number is equivalent to the
non-vanishing of C.~Lescop's generalization~$\lambda$ of the
Casson-Walker invariant, cf.~\cite{Les}.  See T.~Cochran and
J.~Masters~\cite{CM} for generalizations.
\end{remark}

Now \theoref{t:massey} will follow from \theoref{t:self} below.

\begin{theorem}\label{t:self}
Assume~$b_1(M^n)=2$. The non-triviality of the self-linking class
in~$H_{n-3}(M)$ implies the bound~$\syscat(M)\geq 3$.
\end{theorem}

The theorem is immediate from the proposition below.  If the fiber
class in~$M$ of the Abel--Jacobi map vanishes, one can define the
self-linking class of a typical fiber, and proceed as follows.

\begin{prop}
\label{propal}
The non-vanishing of the self-linking of a typical fiber
$\AJ_M^{-1}(p)$ of~$\AJ_M: M \to \T^2$ is a sufficient condition for
the non-vanishing of the fiber class~$\fmanifold$ in the maximal free
abelian cover~$\manbar$ of~$M$.
\end{prop}

\begin{proof}
The argument is modeled on the one found in \cite{KL} in the case
of~$3$-manifolds, and due to A. Marin (see also
\cite[p.~165-166]{SGT}).  Consider the pullback diagram
\begin{equation*}
\CD \manbar @>\ov{\AJ}_M >> \R^2\\ @VpVV @VVV\\ M @>\AJ_M >> \T^2
\endCD
\end{equation*}
where~$\AJ_M$ is the Abel--Jacobi map and the right-hand map is the
universal cover of the torus.  Choose points~$x,y\in \R^2$ with
distinct images in~$\T^2$.  Let~$\ov F_x=\overline{\AJ}_M^{-1}(x)$
and~$\ov F_y= \overline{\AJ}_M^{-1}(y)$ be lifts of the corresponding
fibers~$F_x, F_y \subset M$.  Choose a properly imbedded
ray~$r_y\subset \R^2$ joining the point~$y\in \R^2$ to infinity while
avoiding~$x$ (as well as its~$\Z^2$-translates), and consider the
complete hypersurface
\begin{equation*}
S = \overline{\AJ}^{-1}(r_y) \subset \manbar
\end{equation*}
with~$\partial S = \ov F_y$.  We have~$S\cap T_g \ov F_x = \emptyset$
for all~$g\in G$, where~$G\cong\Z^2$ denotes the group of deck
transformations of the covering~$p: \manbar \to M$ and~$T_g$ is the
deck transformation given by~$g$.

We will prove the contrapositive.  Namely, the vanishing of the class
of the lift of the fiber implies the vanishing of the self-linking
class.  If the surface~$\ov F_x$ is zero-homologous in~$\manbar$, we
can choose a compact hypersurface~$\surface \subset \manbar$ with
\begin{equation*}
\partial \surface = \ov F_x
\end{equation*}
(if there is no such hypersurface for~$\ov F_x$, we work with a
sufficiently high multiple~$N F_x$, see \cite{KL} for details).
The~$(n-3)$-dimensional homology class~$\ell_M(F_x, F_y)$ in~$M$ can
therefore be represented by the~$(n-3)$-dimensional cycle given by the
oriented intersection~$p(\surface) \cap F_y$.  Now we have
\begin{equation}
\label{1321}
\begin{aligned}
p(\surface) \cap F_y= \sum _{g\in G} T_g \surface \cap \ov F_y =
\sum _{g\in G} \partial \left( T_g \surface \cap S \right).
\end{aligned}
\end{equation}
But the last sum is a finite sum of boundaries, and hence represents
the zero homology class.  The finiteness of the sum follows from the
fact that the first sum contains only finitely many non-zero summands,
due to the compactness of~$\surface$.
\end{proof}

To summarize, if the lift of a typical fiber to the maximal free
abelian covering of~$M^n$ with~$b_1(M)=2$ defines a nonzero class, then
one obtains the lower bound~$\syscat(M)\geq 3$, due to the existence
of a suitable systolic inequality corresponding to the partition
\begin{equation*}
n=1+1+(n-2)
\end{equation*}
as in \eqref{eq:main}, by applying Gromov's \theoref{t:abcover}.

%\vfill\eject

\end{document}